\documentclass[12pt,reqno]{amsart}
\usepackage{amssymb}
\usepackage{mathtools}
\usepackage[english,french]{babel}
\usepackage{epsfig}
\usepackage{mathptmx}
\usepackage{amsmath,amsfonts,amssymb}
\usepackage{mathtools}
\usepackage{mathrsfs}
\usepackage[all]{xy}
\usepackage{graphicx}
\usepackage{latexsym}
\usepackage{verbatim}
\usepackage{xcolor}

\newcommand{\hide}[1]{}
\numberwithin{equation}{section}

\def\Im{{\sf Im}\,}

\def\eps{\varepsilon}

\newcommand{\D}{\mathbb D}

\newcommand{\R}{\mathbb R}

\newcommand{\C}{\mathbb C}

\newcommand{\N}{\mathbb N}

\def\Im{{\sf Im}\,}

\def\Im{{\sf Im}\,}

\def\Im{{\sf Im}\,}

\def\1#1{\overline{#1}}
\def\2#1{\widetilde{#1}}
\def\3#1{\widehat{#1}}
\def\4#1{\mathbb{#1}}
\def\5#1{\frak{#1}}
\def\6#1{{\mathcal{#1}}}

\def\Im{{\sf Im}\,}

\newcommand{\mcite}[1]{\csname b@#1\endcsname}

\theoremstyle{theorem}

\def\Im{{\sf Im}\,}

\newtheorem{satz}{Theorem}

\newtheorem{theorem}{Theorem}[section]
\newtheorem{lemma}[theorem]{Lemma}

\theoremstyle{definition}

\newtheorem{example}[theorem]{Example}

\theoremstyle{remark}
\newtheorem{remark}{Remark}

\newtheorem{problem1}[theorem]{Problem}

\numberwithin{equation}{section}

\title[Blow--up Solutions of Liouville's Equation and Quasi--Normality]{Blow--up Solutions of Liouville's Equation and\\[2mm] Quasi--Normality}

\author[J. Grahl]{J\"urgen Grahl}
\address{J. Grahl: Department of Mathematics, University of W\"urzburg, Emil
  Fischer Strasse 40, 97074, W\"urzburg, Germany.}
\email{grahl@mathematik.uni-wuerzburg.de}

\author[D. Kraus]{Daniela Kraus}
\address{D. Kraus: Department of Mathematics, University of W\"urzburg, Emil Fischer Strasse 40, 97074, W\"urzburg, Germany.} \email{dakraus@mathematik.uni-wuerzburg.de}

\author[O. Roth]{Oliver Roth}
\address{O. Roth: Department of Mathematics, University of W\"urzburg, Emil Fischer Strasse 40, 97074, W\"urzburg, Germany.} \email{roth@mathematik.uni-wuerzburg.de}


\dedicatory{Dedicated to the Memory of Professor Stephan Ruscheweyh -- our Teacher, Mentor, and Friend}

\long\def\REM#1{\relax}

\begin{document}
\maketitle

\selectlanguage{english}
\begin{abstract} We prove that
  the  family $\mathcal{F}_C(D)$ of  all meromorphic functions $f$  on a domain
  $D\subseteq \C$ with the property that the spherical
  area of the  image domain $f(D)$  is uniformly bounded by $C
\pi$  is quasi--normal of order $\le C$. We also discuss the close relations between
this result and the 
well--known work of Br\'ezis and Merle on blow--up solutions of Liouville's
equation. These results are completely in the spirit of Gromov's compactness
theorem, as pointed out at the end of the paper.
\end{abstract}

\renewcommand{\thefootnote}{\fnsymbol{footnote}}
\setcounter{footnote}{2}

\section{Introduction}

 In their celebrated paper \cite{BrezisMerle},  Br\'ezis and Merle have
 pioneered the  study of bubbling phenomena for solutions of semilinear elliptic PDEs
 with exponential nonlinearity in two dimensions. 
At the core of their work is the ``compactness--concentration'' principle,
which roughly says  that
a loss of compactness of solutions of the 
twodimensional nonlinear equation 
$$
-\Delta u = V (z)e^{2u} \, .
$$
implies the existence of finitely many blow-up points (bubbles):

\begin{satz}[Br\'ezis--Merle \cite{BrezisMerle}]
  Let $D\subseteq \C$ be a bounded domain  and $p \ge 1$. Suppose that $(u_n)$ is a sequence of (weak) solutions of
  \begin{equation} \label{eq:pde1}
-\Delta u_n = V_n (z)e^{2u_n} \quad \text{ in } D \, , 
    \end{equation}
with $ 0 \le V_n \le C_1$ and $||e^{2u_n}||_{L^p(D)} \le
    C_2$    for some constants $C_1>0$ and $C_2>0$. Then -- after taking a
    subsequence -- one of the following alternatives holds:
    \begin{enumerate}
    \item either $(u_n)$ is locally bounded in $D$ or $u_n \to -\infty$ locally uniformly in $D$;
\item there exists a finite nonempty set $S \subseteq D$ with the
  following properties:
\begin{itemize}
\item[(2a)] (Bubbling)\\
  $u_n \to
  -\infty$ locally uniformly in $D \setminus S$ and for each $p \in S$ there
  is a sequence $(z_n) \in D$ such that $z_n \to p$ and $u_n(z_n) \to
  +\infty$. 
\item[(2b)] (Mass Concentration)\\For each $p \in S$ there is  $\alpha_p \ge 1$ such that in the
measure theoretic sense
$$\frac{1}{\pi} e^{2 u_n}  \to  \sum \limits_{p \in S} \alpha_p \delta_{p} \,
,   $$
 that is,
\begin{equation} \label{eq:conv}
\frac{1}{\pi} \int \limits_D e^{2 u_n(z)} \psi(z) \, dx dy \to \sum \limits_{p
\in S} \alpha_p \psi(p) 
\end{equation}
for any continuous function $\psi : D \to \R$ with compact support in $D$.
\end{itemize}
  \end{enumerate}
\end{satz}

We focus on the constant case $V_n=4$ which is related to numerous geometric
and physical problems, see Section \ref{subsec:liouville}. In this case,  Theorem A is in reality a
    result about \textit{locally univalent} meromorphic functions. This
    follows from  a classical result of Liouville \cite{Lio1853} which asserts that every solution to $-\Delta u=4 e^{2u}$ has (locally) the form
  $$ u(z)=\log f^{\sharp}(z)$$
  for some locally univalent meromorphic function $f$  and vice versa. Here,
  as it is standard, $f^{\sharp}$ denotes the spherical derivative of $f$
  defined by
  $$ f^{\sharp}(z):=\frac{|f'(z)|}{1+|f(z)|^2}\, .$$


  In this note we  prove:

\begin{theorem} \label{thm:main}
  Let $D \subseteq \C$ be a domain and $C>0$. Denote by $\mathcal{F}_C$  the
  set of all functions $f$ meromorphic on $D$ such that
  \begin{equation} \label{eq:family}
  \frac{1}{\pi} \iint \limits_{D} \big( f^{\sharp}(z) \big)^2 \, dxdy \le C \, .
\end{equation}
Then for every sequence $(f_n)$ in $\mathcal{F}_C$ -- after taking a
subsequence -- there is $f \in \mathcal{F}_C$ such that one of the following alternatives hold:
\begin{enumerate}
\item  $(f_n)$  converges locally uniformly in $D$ to $f$ (w.r.t.~the
  spherical metric);
     \item There exists a finite nonempty $S \subseteq D$ with  at
       most $C$ points with the following properties:
       \begin{itemize}
         \item[(2a)] (Bubbling)\\
     $(f_n)$  converges locally uniformly in $D \setminus S$  to $f$ and for each $p \in S$ there
  is a sequence $(z_n) \in D$ such that $z_n \to p$ and $f^{\sharp}_n(z_n) \to
  +\infty$. If each $f_n$ is locally univalent, then $f$ is constant. 
\item[(2b)] (Mass Concentration)\\
  For each $p \in S$ there is a real number $\alpha_p \ge 1$ such that in the
  measure theoretic sense
  $$ \frac{1}{\pi} \big(f_n^{\sharp}\big)^2 \to \sum \limits_{p \in S} \alpha_p \delta_p+
  \frac{1}{\pi} \big(f^{\sharp}\big)^2 \, .$$
  \end{itemize}
\end{enumerate}
  \end{theorem}

  Condition (\ref{eq:family}) means that the spherical area of the image
  domain $f(D)$ on the Riemann sphere $\hat{\C}$ is $\le C \pi$ (counting multiplicities and with the normalization that the area of  $\hat{\C}$ is $=\pi$).

\section{Remarks and Questions}

\subsection{Theorem \ref{thm:main} vs.~Theorem A}

 Theorem \ref{thm:main}
    extends the result of Br\'ezis--Merle \cite{BrezisMerle}  (for $V_n=4$, $p=1$) 
    from the case of locally univalent meromorphic functions to all
    meromorphic functions in $\mathcal{F}_C$. This follows from Liouville's
    theorem mentioned above and Hurwitz'
    theorem, which says that the limit function of a locally uniformly
    convergent sequence of locally
    univalent meromorphic functions is either (i) locally univalent or (ii) constant.
In case (i), $(\log f^{\sharp}_n)$ is locally uniformly bounded in $D$; in
case (ii), $\log f^{\sharp}_n\to -\infty$ locally uniformly in $D$.

\subsection{Bubbling \& Quasi--Normality}
  Recall 
  that a family $\mathcal{F}$ of meromorphic functions on a domain $D$ is
  called quasi--normal if every sequence in $\mathcal{F}$ has a subsequence
  which converges locally uniformly with respect to the spherical metric on $D
  \setminus S$ where the set $S$ (which may depend on the extracted
  subsequence) has no accumulation point in $D$. If  $S$ has always at most $\nu \ge 0$ points, then $\mathcal{F}$ is said to
  be quasinormal of order $\le \nu$.  In particular, Theorem \ref{thm:main}
  says that ${\mathcal F}_C$ is quasi--normal of order at most $C$.
  The notion of quasi--normality generalizes the fundamental concept of
  normal families (= quasi--normal families of order zero), and 
has been introduced by Montel \cite{Montel1922} as early as 1922. We refer
  to \cite[Appendix]{Sc} for background on  quasi--normal families.

\subsection{No Bubbling:  Small area and bounded weighted area} \label{sec:nobub}
\begin{itemize}
  \item[(i)]
If $C<1$, then Theorem \ref{thm:main} says that $\mathcal{F}_C$ is a compact
family. For the case that $D$ is the unit disc $\D:=\{z \in \C \, : |z|<1\}$,
this was already observed by Montel (1934, \cite{Montel1934}); if $D$ is an
arbitrary (hyperbolic) domain this result is stated explicitly in 
\cite[Theorem 2]{Yamashita2000}.
\item[(ii)]
 If $V_n=4$ and $p>1$ in Theorem A, then  there are no bubbles. This follows from the work of
 Aulaskari and Lappan (1988, \cite{AulaskariLappan}), who have shown that for each fixed $C>0$
 and $s>2$ the family of all
 meromorphic functions $f$ in $\D$  satisfying
 \begin{align} \label{eq:exponent}
   \frac{1}{\pi} \int \limits_{\D} \big(f^{\sharp}(z)\big)^s \, dx dy \le C
 \end{align}
 is a normal family. Hence Theorem \ref{thm:main} handles the borderline case
 $s=2$ where bubbles actually occur (take e.g.~$f_n(z)=nz$), but at most
 finitely many. 
 \item[(iii)] Reducing the exponent $s$ in (\ref{eq:exponent}) below the
   crucial threshold $s=2$ does not even guarantee quasi--normality. A simple
   example is given by the family $\mathcal{F}=\{f_n(z):=e^{inz} \, : \, n=1,2, \ldots\}$. It is
   easy to show that (\ref{eq:exponent}) holds for $s=1$ with $C=1$ for all
   $f \in \mathcal{F}$, but $f_n^{\sharp}(z) \to +\infty$ as $n \to \infty$ for every $z \in \D$
   with $\Im z=0$.  Hence every point in the interval $S:=(-1,1)$ is a
   ``bubble''. In particular, $\mathcal{F}$ is not quasi--normal on $\D$. Note
   that $(f_n)$ is normal in $\D \setminus S$ and in fact converges to a 
   constant limit function on each of the components of $\D \setminus S$.
   \begin{problem1} Suppose $1<s<2$ and $\mathcal{F}$  is a family of meromorphic functions in $\D$
    satisfying (\ref{eq:exponent}). It seems likely that $\mathcal{F}$ is not
    necessarily   quasi--normal in $\D$. Is it
    possible to quantify the maximal size of possible exceptional sets $S$  of
    $\mathcal{F}$ in terms of the exponent $s$\,?
     \end{problem1}
  \end{itemize}

  \subsection{No Bubbling: Local univalence} \label{NBLU}
The intersection of Theorem A  and Theorem
\ref{thm:main} deals with the case of all  \textit{locally univalent}
functions in $\mathcal{F}_C$. In fact, in this ``locally univalent'' case,  there
is the following general ``no bubbles'' criterion:
For a  domain $D\subseteq \C$ and a family $\mathcal{L}$  of locally univalent meromorphic functions on $D$, the following are equivalent:
  \begin{itemize}
    \item[(i)] $\mathcal{L}$ is compact;
     \item[(ii)] For each compact set $K
    \subseteq D$ there is a positive constant $c>0$ such that
    $$ c \le f^{\sharp}(z)  \quad \text{ for all } f \in \mathcal{L}
    \text{ and all } z \in K \ .$$
  \end{itemize}
 The non--trivial implication is (ii) $\Longrightarrow$ (i) and follows from the
general fact  that each of the families
$$ \mathcal{L}_c(D):=\{f \text{ meromorphic in } D \, : \, f^{\sharp}(z) \ge c \text{ for all }
z \in D \} \, , \qquad c>0 \, ,$$
is compact (see
\cite[Corollary 8]{BrezisMerle} and \cite{GN}).
 Hence compactness of $\mathcal{L}$ just means that
the family $\{f^{\sharp} \, : \, f \in \mathcal{L}\}$ is locally uniformly bounded away
from zero. This fact plays a crucial role in our proof of Theorem
\ref{thm:main}{. The set $\mathcal{L}_c(D)$ is also interesting in its own right and has been
studied  e.g.~in \cite{BrezisMerle,FKR,GN,Janne,Sha,St}; see also Section \ref{NBLU2} below. 

\begin{remark}
It is perhaps a bit surprising that $\mathcal{L}_c(\D)$ is not empty only if $c \le
1/2$, see \cite{GN}. The following beautiful proof was shown to us some years ago by Stephan (see also
\cite[Proof of Theorem 1.2]{FRold}): Let $f \in \mathcal{L}_c(\D)$. Postcomposing $f$
with a rigid motion of the sphere does not change the spherical derivative, so
we may assume w.l.o.g.~that $f(0)=0$.
Then $f(z)/(z f'(z))$ is {\it holomorphic} in $\D$ with value $1$ at $z=0$. Hence the maximum
principle implies
$$ 1 \le \max \limits_{|z|=\rho} \left| \frac{f(z)}{z f'(z)} \right|=
\frac{1}{\rho}  \, \max \limits_{|z|=\rho} \left(
  \frac{|f(z)|}{1+|f(z)|^2} \frac{1}{f^{\sharp}(z)} \right) \le 
\frac{1}{c \rho}  \, \max \limits_{|z|=\rho} 
 \frac{|f(z)|}{1+|f(z)|^2} \le \frac{1}{2c\rho} \, , \quad 0<\rho<1 \, ,$$
so $c \le 1/2$. It is now obvious that if $D$ contains a disk of radius $R>0$,
then $\mathcal{L}_c(D)$ is possibly not empty only if $c \le 1/(2 R)$.
  \end{remark}

\subsection{Quantification and Schwarz Lemmas: Small Area}
Let  $\mathcal{F}$ be  a normal family of meromorphic functions on a domain
$D\subseteq \C$.  Then, by Marty's criterion, the quantity
$$ M_{\mathcal{F}}(z):=\sup \limits_{f \in \mathcal F} f^{\sharp}(z)$$
is finite for each $z \in D$. Geometrically, $M_{\mathcal{F}}(z)$  provides the maximal spherical--euclidean
distortion of a function $f \in \mathcal{F}$ at the point $z$. Finding
$M_{\mathcal{F}}$ for a given family $\mathcal{F}$ is an ubiquitous task in
  Geometric Function Theory, and quite often a challenging endeavour.
  
As in Theorem \ref{thm:main}, we  focus on the families $\mathcal{F}_C$, for which we now write $\mathcal{F}_C(D)$ in order to emphasize the dependence on the domain $D$. 
If $C<1$, then
$$ M_{\mathcal{F}_C(\D)}(z) \le \sqrt{\frac{C}{1-C}} \,\frac{1}{1-|z|^2} \, , \qquad
z \in \D \, .$$
This is essentially a classical result of Dufresnoy (1941,
\cite{Dufresnoy1941}), see also \cite[p.~83]{Sc}. The result of Dufresnoy has
been extended from the disk $\D$ to arbitrary hyperbolic domains $D$ by Yamashita
(2000, \cite{Yamashita2000}), who showed that for any hyperbolic domain $D
\subseteq \C$, 
$$ M_{\mathcal{F}_C(D)}(z) \le \sqrt{\frac{C}{1-C}} \, \lambda_{D}(z) \, , \qquad
z \in D \, , $$
with equality for one -- and then for any -- point $z \in D$ if and only if
$D$ is simply connected.
Here, $\lambda_D(z)$ is the density of the Poincar\'e metric on $D$ (with
curvature $-4$).

In the spirit of Schwarz' lemma, these results say: for any $C<1$ each map $f \in \mathcal{F}_C(D)$ is a
Lipschitz-map from $D$ into the Riemann sphere $\hat{\C}$, provided  both are equipped with
their ``natural'' geometries (hyperbolic resp.~spherical).

\subsection{Quantification and Schwarz Lemmas: The problem of Br\'ezis--Merle} \label{NBLU2}
 Let 
$\mathcal{L}_c(D)$ be the family of all meromorphic functions on $D$ with spherical
derivative bounded from below by $c>0$. Recall from Section \ref{NBLU} that $\mathcal{L}_c(D)$ is
compact. Suppose that $K$ is a compact subset of $D$. Br\'ezis \& Merle
\cite[p.~1239, Open problem 3]{BrezisMerle} have posed, amongst others, the problem to find
$$ \max \limits_{f \in \mathcal{L}_c(D), \, z \in K} f^{\sharp}(z) \, .$$
It seems that little is known about this problem. Shafrir \cite{Sha} has proved
that
$$ \sup \limits_{z \in K} f^{\sharp}(z) \le \frac{e^{C_2}}{c} \, , \qquad f
  \in \mathcal{L}_c(D) \, , $$
  where $C_2>0$ is a constant depending only on $K$.
In the case $D=\D$  Steinmetz \cite{St} has quantified Shafrir's estimate by proving that
\begin{equation} \label{eq:stein}
\sup \limits_{f \in \mathcal{L}_c(\D)} f^{\sharp}(z) \le  \frac{1}{c}
\frac{1}{(1-|z|^2)^2} \, , \qquad z\in \D \, .
\end{equation}
 In \cite{FKR} this has  been improved to
\begin{equation} \label{eq:FKR}
 \sup \limits_{f \in \mathcal{L}_c(\D)} f^{\sharp}(z) \le \frac{1+\sqrt{1-4 c^2 \left(1-|z|^2 \right)^2}}{2 c
   \left(1-|z|^2 \right)^2} \, , \quad z \in \D \, ,
 \end{equation}
  showing that the dependence of the upper bound  on the parameter $c$ found by Shafrir
  resp.~Steinmetz is not best possible. Asymptotically however, both
  estimates (\ref{eq:stein}) and (\ref{eq:FKR}) yield 
\begin{equation} \label{eq:5a} \limsup \limits_{|z| \to 1} \left(1-|z|^2 \right)^2 \,
  f^{\sharp}(z)  \le \frac{1}{c} \, , \qquad f \in \mathcal{L}_c(\D) \, . 
\end{equation}
Recent work of Gr\"ohn \cite[Theorem 3]{Janne}, which is based in parts on Carleson's celebrated solution of the
$H^{\infty}$--interpolation problem, shows that there is a function $f
\in \bigcup_{c>0}\mathcal{L}_c(\D)$ such that
$$
\inf \limits_{n \in \N}  \left(1-|z_n|^2 \right)^2 \, f^{\sharp}(z_n)>0
\,
$$
for some sequence $(z_n)$ in $\D$ with $|z_n| \to 1$. Hence, for sufficiently small values of $c>0$ inequality (\ref{eq:5a}) is
sharp up to a multiplicative constant.  It is shown in \cite{FKR} that for all possible
values of $c$ one can replace the
number $1$ on the right--hand side of (\ref{eq:5a}) by $(3-\sqrt{5})/2 \approx
0.38$.

We finally note that a complete solution of
the Br\'ezis--Merle problem in the very special case $D=\D$ and $K=\{0\}$ has
been given in \cite{FKR}, where it is shown that
  $$ \max \limits_{f \in \mathcal{L}_c(\D)} f^{\sharp}(0)=\frac{1+\sqrt{1-4c^2}}{2c} \,
  .$$
The problem to determine  $\max \limits_{f \in \mathcal{L}_c(\D)} f^{\sharp}(z)$ for $z
\not=0$ remains open.

\subsection{Mass Quantisation}

It has been conjectured by Br\'ezis and Merle \cite[Open Problem 4]{BrezisMerle} that under
some mild regularity assumptions on the functions $V_n$ in Theorem A all the
numbers $\alpha_p$ are in fact positive \textit{integers} (Mass
Quantisation). This conjecture has been confirmed in \cite{LS} under the condition 
that  $V_n \in C(\overline{D})$ and $V_n \to V$ in $C(\overline{D})$.
It therefore seems natural to expect that also in Theorem \ref{thm:main} each $\alpha_p$ is an integer.
If in Theorem \ref{thm:main} each $f_n$ is locally univalent, then this
follows at once from \cite{LS}.

  \subsection{Remarks on Liouville's equation} \label{subsec:liouville}
 As mentioned above, Liouville's equation $-\Delta u=4 e^{2u}$ can be seen as
 the ``governing'' PDE of the set of locally univalent meromorphic
 functions. Despite, or maybe because of this, it also  arises in a variety of diverse problems in analysis, geometry and physics. 
  For instance, if  $D$ is  a bounded domain in $\R^2$ , then the Euler--Lagrange equation for the functional
\begin{equation} \label{eq:MT}
  J(v)=\frac{1}{2} \int \limits_{D} |\nabla v|^2 -8 \pi \log \int
  \limits_{D} e^v \, , \qquad v \in W^{1,2}_0(D) \, ,
\end{equation}
which is closely tied to  the well--known Moser-Trudinger inequality,
is the Liouville--type equation
\begin{equation} \label{eq:EL} 
  -\Delta v=\lambda \frac{e^v}{\int \limits_{D} e^v} \, ,  
  \end{equation}
for some constant $\lambda > 0$. The functional
(\ref{eq:MT}) and equation (\ref{eq:EL}) have been intensively studied, partly
in view of many applications such as  the problem
of prescribing Gauss curvature \cite{CGY,CY,CD},  the theory of the mean field equation
\cite{DJLW,DJLW2,CL,CL2} and  Chern-Simons theory
\cite{SY,ST,T,DJLW3,N}.

  \section{Quasi--normality and wandering exceptional functions}

  An essential step in the proof of Theorem \ref{thm:main} is to show that the
  family $\mathcal{F}_C$ is {\it quasi--normal}.   In order to prove this one
  can apply a quasi--normality criterion which has already been
  established by Montel \cite{Montel1924} and Valiron \cite{Valiron1929}. 
  The Montel--Valiron criterion  might be viewed as  a ``quasi--normal'' version of Carath\'eodory's \cite[p.~104]{Sc}
  extended  Fundamental Normality Test (FNT)   for a family $\mathcal{F}$ of
  meromorphic functions.
  In fact, we prove here an extension of the result of Montel and Valiron, which shows that one can replace
the exceptional {\it values} in the Montel--Valiron criterion by 
  exceptional {\it meromorphic functions} ``wandering'' on the sphere. 
Here ``wandering'' means that the exceptional functions are allowed to depend
on the individual members of the family $\mathcal{F}$.

\medskip

We first need to recall some standard terminology.
$\mathcal{M}(D)$ denotes the set of all meromorphic functions on $D$
and $\sigma$ denotes the spherical distance on $\hat{\C}$. A function $f$ is understood to assume the function $a$ if there is a point
$z_0\in D$ such that $f(z_0)=a(z_0)$. We call such a point  $z_0$ an $a$-point of $f$. In the case $f(z_0)=a(z_0)=\infty$ it
isn't assumed that $z_0$ is a zero of $f-a$. (Actually, for $a\equiv\infty$ the latter wouldn't make sense.) Furthermore, we do 
not take multiplicities into account, i.e. if (in the finite case) $z_0$ is a
multiple zero of $f-a$, it is counted only once.

\begin{theorem}\label{3ExceptionalFunctionsQuasiNormal}
Let $\mathcal{F}$ be a family of functions meromorphic on a domain $D\subseteq
\C$,
$\varepsilon>0$ and $p,q,r$ non-negative integers with $p\le q\le
r$. Assume that for each $f\in\mathcal{F}$ there exist functions
$a_f,b_f,c_f\in\mathcal{M}(D)\cup\{\infty\}$ such that $f$ assumes the
function $a_f$ at most $p$ times, the function $b_f$ at most $q$ times
and the function $c_f$ at most $r$ times in $D$ (ignoring multiplicities), and such that
\begin{equation}\label{DistanceBoundedQ}
\min\{\sigma(a_f(z),b_f(z)),
\sigma(a_f(z),c_f(z)),\sigma(b_f(z),c_f(z))\}\ge \varepsilon 
\end{equation}
for all $z\in D$. Then $\mathcal{F}$ is a quasi-normal family of order at
most $q$. 
\end{theorem}

\begin{remark}[Exceptional Values vs.~Exceptional Functions and  Wandering vs.~Non--wandering]
Normality criteria in the spirit of Montel's or Carath\'eodory's FNT, but with exceptional functions instead of exceptional values,
have been established by Bargman, Bonk, Hinkkanen and Martin \cite{BBHM} as
well as by the first named author and Nevo \cite{GN-Wandering}. In \cite{BBHM}
a version of Montel's FNT for exceptional {\it continuous}, but not wandering
functions with disjoint graphs is
proved, while in \cite{GN-Wandering} a version of Carath\'eodory's extended FNT
for exceptional {\it meromorphic} functions wandering on the sphere is established. Theorem \ref{3ExceptionalFunctionsQuasiNormal}
extends the main result of \cite{GN-Wandering} to the case of {\it
  quasi--normal} families. It extends as well the result of Montel and Valiron
mentioned earlier, which  is just the special case of Theorem
\ref{3ExceptionalFunctionsQuasiNormal} that the exceptional functions are
constants, but might be wandering.
Even though it is tempting,  it is not possible to replace the {\it meromorphic} exceptional
functions by {\it continuous} exceptional functions in Theorem
\ref{3ExceptionalFunctionsQuasiNormal}. In fact, for merely continuous wandering exceptional
functions we neither have quasi-normality nor $Q_\alpha$-normality for
any ordinal number $\alpha$. (For the exact definition of
$Q_\alpha$-normality we refer to \cite{Nevo-Qalpha}.)  This is shown
by the same counterexample as in \cite{GN-Wandering} or Section
\ref{sec:nobub} (iii): The functions
$f_n(z):=e^{nz}$ omit the three continuous functions $a_n:\equiv 0$,
$b_n:\equiv\infty$, $c_n(z):=-e^{i n \Im(z)}$ which clearly satisfy
(\ref{DistanceBoundedQ}), but all points on the imaginary axis are
points of non--normality, so in the unit disk $(f_n)_n$ is neither
quasi-normal nor $Q_\alpha$-normal for any ordinal number $\alpha$.
\end{remark}

\section{Proofs}

\subsection{Proof of Theorem \ref{3ExceptionalFunctionsQuasiNormal}}
 
An essential tool in the proof of Theorem
\ref{3ExceptionalFunctionsQuasiNormal} is the following normality
criterion from \cite[Theorem~2]{GN-Wandering}. It is concerned  with pairs of
meromorphic functions
``wandering'' on the sphere which ``uniformly stay away from each other''. 

\begin{lemma} \label{exceptfunctnormal}
Let $\mathcal{G}$ be a family of pairs of meromorphic
functions on a domain $D$ and $\varepsilon>0$. Assume that
\begin{equation}\label{StayAway}
\sigma(a(z),b(z))\ge \varepsilon \qquad\mbox{ for all } (a,b)\in\mathcal{G}
\mbox{ and all } z\in D.
\end{equation}
Then the families $\{ a\;|\, (a,b)\in\mathcal{G} \}$ and $\{ b\;|\,
(a,b)\in\mathcal{G}\}$ are normal on $D$.  
\end{lemma}

Using this lemma, Theorem \ref{3ExceptionalFunctionsQuasiNormal} can
be deduced from \cite[Theorem 1]{GN-Wandering}  basically in the
same way as the well--known special case of Theorem
\ref{3ExceptionalFunctionsQuasiNormal} that $a_f,b_f,c_f$ are constant 
(see \cite[Theorem~A.5 and Theorem A.9]{Sc})
is deduced from the FNT.
In  order to make the paper self-contained and since the proof in \cite[Theorem~A.5]{Sc} is given only for the
special case of analytic functions and fixed exceptional values not
depending on $f$,
we provide full details.

In what follows we denote by $K_r(z_0):=\{z \in \C \, : \, |z-z_0|<r\}$ the
euclidean open disk with center $z_0 \in \C$ and radius $r>0$.
For a quasi--normal sequence $(f_n)$ in $\mathcal{M}(D)$ a point $p \in D$ is
called an irregular point for $(f_n)$ if the sequence $(f_n)$ fails to be normal at $p$.

\begin{proof}[Proof of Theorem \ref{3ExceptionalFunctionsQuasiNormal}]
Let $(f_n)_n$ be a sequence in $\mathcal{F}$. We write $a_n:= a_{f_n}$,
$b_n:= b_{f_n}$, $c_n:= c_{f_n}$.

By Lemma \ref{exceptfunctnormal} the families $\{ a_f\,|\,
f\in\mathcal{F}\}$, $\{ b_f\,|\, f\in\mathcal{F}\}$ and $\{ c_f\,|\,f\in\mathcal{F}\}$
are normal. Therefore there exists a subsequence of
$((f_n,a_n,b_n,c_n))_n$ which we continue to denote by
$((f_n,a_n,b_n,c_n))_n$ such that $(a_n)_n$, $(b_n)_n$ and $(c_n)_n$
converge locally uniformly in $D$ (with respect to the spherical
metric) to limit functions $a,b,c\in\mathcal{M}(D)\cup\{\infty\}$. 

After heavily extracting further subsequences we may assume that there are
at most $p$ accumulation points of the $a_n$-points of $f_n$. More
precisely: We can assume that there is a set $A\subset D$ consisting
of at most $p$ points such that the following holds:
\begin{itemize}
\item[1.]
For each $z_0\in A$ there exists a sequence $(z_k)_k$ in $D$
converging to $z_0$ and a subsequence $(f_{n_k})_k$ such that
$f_{n_k}(z_k)=a_{n_k}(z_k)$ for all $k$.
\item[2.]
For each $z_0\in D\setminus A$ there exists an $r>0$ such that the disk
$K_r(z_0)$ contains only finitely many points $z$ which are
$a_n$-points of $f_n$ for some $n$. 
\end{itemize}
In the same way we find a set $B\subset D$ consisting of at most $q$
points and a set $C\subset D$ consisting of at most $r$ points such
that, in the sense described above, the points in $B$ and $C$,
respectively, are the only accumulation points of the $b_n$-points and
of the $c_n$-points, respectively, of $f_n$.

Set $E:=A\cup B\cup C$. In each compact subset of $D\setminus E$ there are
only finitely many $a_n$-points, $b_n$-points and $c_n$-points of the
functions $f_n$. Therefore, by \cite[Theorem 1]{GN-Wandering},
$(f_n)_n$ is normal in $D\setminus E$, so by moving to a further subsequence
we can assume that $(f_n)_n$ converges locally uniformly in $D\setminus E$
(with respect to the spherical metric) to a limit function
$F\in\mathcal{M}(D\setminus E)\cup\{\infty\}$. 

This shows that there are at most $p+q+r$ irregular  points of
$(f_n)_n$. In order to improve this bound we use the well-known fact
that, by the maximum principle, the locally uniform convergence of a
sequence of {\it analytic} functions  on a punctured disk implies the
locally uniform convergence on the whole disk.

So let some point $z_0\in E$ be given. Then there exists a rigid
motion $T$ of the Riemann sphere (i.e. an isomorphism with respect to
the spherical metric) such that $T(a(z_0))\ne\infty$,
$T(b(z_0))\ne\infty$ and $T(c(z_0))\ne\infty$. Since the normality of
$(f_n)_n$ on a subdomain of $D$ or at a point in $D$ is equivalent to
the normality of $(T\circ f_n)_n$ (due to the fact that the spherical
derivative is invariant under post-composition with rigid motions of
the sphere and due to Marty's theorem) and since (\ref{DistanceBoundedQ})
as well as the other assumptions of our theorem are invariant under
rigid motions, we can replace $((f_n,a_n,b_n,c_n))_n$ by $((T\circ
f_n,T\circ a_n, T\circ b_n,T\circ c_n))_n$ if required. Therefore,
without loss of generality we may assume that $a(z_0)\ne\infty$,
$b(z_0)\ne\infty$ and $c(z_0)\ne\infty$. Now we consider two cases.

{\bf Case 1:} $F\not\equiv a$

Assume that $z_0\not\in A$. We choose $R>0$ such that the punctured
disk $K_{3R}(z_0)\setminus\{ z_0\}$ is contained in $D\setminus E$ and such that
$a$ is analytic in $K_{3R}(z_0$). Then $K_{3R}(z_0)\cap A=\emptyset$,
hence all but finitely many $f_n$ omit $a_n$ in $K_{2R}(z_0)$ since
otherwise the points where $f_n$ assumes $a_n$  would have
an accumulation point in $K_{3R}(z_0)$, contradicting the definition
of $A$. This means that $g_n:=\frac{1}{f_n-a_n}$ is analytic in
$K_{2R}(z_0)$ for all but finitely many $n$.  In the compact annulus
$K=\overline{K_{2R}(z_0)}\setminus K_R(z_0)$ the sequence $(g_n)_n$ converges
uniformly to $G:=\frac{1}{F-a}\not\equiv\infty$.  By the maximum
principle, the uniform convergence carries over to the whole disk
$K_{2R}(z_0)$. So $(f_n)_n$ itself converges uniformly (with respect
to the spherical metric) in this disk to $\frac{1}{G}+a$.

We conclude that the only possible irregular points of $(f_n)_n$ are
the points in $A$. In particular, there are at most $p\le q$ such
points. 

{\bf Case 2:} $F\equiv a$

Assume that $z_0\not\in B$. Then as in Case 1 we find an $R>0$ such that
$g_n:=\frac{1}{f_n-b_n}$ is analytic in $K_{2R}(z_0)$
for all but finitely many $n$ and such that $(g_n)_n$ converges
uniformly in the compact annulus $K=\overline{K_{2R}(z_0)}\setminus K_R(z_0)$ to
the analytic limit function $G:=\frac{1}{F-b}=\frac{1}{a-b}$. Again,
by the maximum principle we deduce the uniform convergence of this
sequence in the whole disk $K_{2R}(z_0)$, hence the uniform
convergence of $(f_n)_n$ in this disk.

So in this case irregular points of $(f_n)_n$ can occur only at the
points of $B$. Therefore their number is bounded by $q$. 
\end{proof}

\subsection{Proof of Theorem \ref{thm:main}}

We start with the following elementary lemma.

\begin{lemma} \label{Lem:PosAreaThreePoints}
Let $\alpha>0$ be a real number. Then there exists a $\delta>0$ such that
any measurable subset $E$ of the Riemann sphere with spherical area at
least $\alpha$ contains three points $a,b,c\in E$ such that
$$\min\{\sigma(a,b),\sigma(a,c),\sigma(b,c)\}\ge \delta.$$
\end{lemma}

\begin{proof}
Without loss of generality, and to simplify our considerations, we replace the sphere by
the euclidean plane $\C$ and the spherical metric by the euclidean metric.
Let $E\subseteq \C$ be some fixed measurable set with euclidean area
$\ge\alpha$.
First, we can find two points $a,b\in E$ such that
$$|a-b|\ge \sqrt{\frac{\alpha}{\pi}}=:r\, .$$ (Otherwise, $E$
would be contained in an euclidean  disc of radius less than $r$, and hence of area
less than $\pi r^2=\alpha$.) Then, we can find a third point
$c\in E$ such that $$\min\{ |a-c|,|b-c|\}\ge
\frac{r}{\sqrt{2}}$$ since otherwise $E$ would be contained in the
union of two discs centered at $a$ and $b$, respectively, and with
radii less than $r/\sqrt{2}$, and the area of the union of these discs would
be less than $$2\pi \cdot\left(\frac{r}{\sqrt{2}}\right)^2=\pi
r^2=\alpha\, .$$  So the assertion of our lemma holds with
$\delta:=r/\sqrt{2}$.
\end{proof}

We are now in a position to prove

\begin{theorem} \label{IntCrit}
  Let $D \subseteq \C$ be a domain, $C>0$, and $\mathcal{F}$ be a family
  of functions meromorphic on $D$ such that
\begin{equation}\label{IntCond}
\frac{1}{\pi}\int \limits_{D} (f^\#(z))^2\,dx\,dy\le C \qquad \mbox{ for all }
f\in\mathcal{F}\, .
\end{equation}
Then $\mathcal{F}$ is quasi-normal of order at most $C$.
\end{theorem}

\begin{example}
  The estimate for the order of quasi-normality in
Theorem \ref{IntCrit} is best possible. This is illustrated by the
family of the functions $f_n:= n\cdot P$ on the unit disk $\D$ where $P$ is an
arbitrary polynomial of degree $m$ with exactly $m$ distinct zeros in
the unit disk (for example, $P(z)=2z^m-1$). Clearly, $(f_n)$ is
quasi-normal of order $m$, the irregular points being the zeros of
$P$. Since $f_n$ assumes each value in $\C$ at most $m$ times,
(\ref{IntCond}) holds with $C=m$.
\end{example}

\begin{proof}[Proof of Theorem \ref{IntCrit}]
Recall that the integral $$\int \limits_{D} \big(f^\#(z)\big)^2\,dx\,dy$$
is the area of the image domain $f(D)$ of $D$ under the mapping $f$ on
the Riemann sphere $\hat{\C}$, determined with respect to multiplicities, while
$\pi$ is the area of the whole Riemann sphere. Let $m$ denote the largest
integer less or equal than $C$ and
let  $$\varepsilon:=\pi\cdot\left( 1-\frac{C}{m+1}\right)>0\, .$$
Then
for each $f\in\mathcal{F}$ there is a set $E_f\subseteq \hat{\C}$ of spherical measure at
least $\varepsilon$ such that all values in $E_f$ are assumed by $f$
at most $m$ times. Otherwise, all values in a set of spherical measure
larger than $\pi-\varepsilon$ would be assumed at least $m+1$ times,
which would imply
$$\int_{D} (f^\#)^2(z)\,dx\,dy >
(m+1)(\pi-\varepsilon)=\pi \cdot C,$$
contradicting (\ref{IntCond}).
 By Lemma
\ref{Lem:PosAreaThreePoints} for each $f\in\mathcal{F}$ we can find for each
$f\in\mathcal{F}$ three points $a_f,b_f,c_f\in E_f$ such that 
$$\min\{\sigma(a_f,b_f),\sigma(a_f,c_f),\sigma(b_f,c_f)\}\ge
\delta$$
with an universal constant $\delta>0$ depending only on $\varepsilon$,
but not on the function $f\in\mathcal{F}$. Since each $f\in\mathcal{F}$ assumes the
values $a_f,b_f,c_f$ at most $m$ times,  the family  $\mathcal{F}$ is
quasi-normal of order at most $m$ by Theorem
\ref{3ExceptionalFunctionsQuasiNormal}
\end{proof}

By  the converse of Bloch's principle \cite[Chapter 4]{Sc},
Theorem \ref{IntCrit} should encapsulate  a
corresponding removable singularity criterion. Such a criterion is in fact a simple consequence of Picard's Great Theorem:

\begin{lemma} \label{lem:heb}
Let $f$ be meromorphic on the punctured disk $K_R(z_0) \setminus \{z_0\}$ such
that
\begin{equation} \label{eq:remove}
 \int \limits_{K_R(z_0)\setminus \{z_0\}} \big(f^{\sharp}(z)\big)^2 \, dx dy <\infty
 \, .
 \end{equation}
  Then $f$ has a meromorphic extension to $K_R(z_0)$.
  \end{lemma}

In order to see how this follows from Picard's Great Theorem note that 
  (\ref{eq:remove}) is equivalent to the convergence of the series 
  $$ \sum \limits_{n=1}^{\infty} \int \limits_{\frac{R}{2^{n}}<|z-z_0|<
    \frac{R}{2^{n-1}}}  \big(f^{\sharp}(z)\big)^2 \, dx dy \, .$$
Hence,  for sufficiently small $\rho>0$, the function $f$ restricted to $K_{\rho}(z_0)\setminus
\{z_0\}$ would omit a set of spherical area less than $\pi/2$, say, so $f$
would omit there (at least) three values, so cannot have an essential
singularity at $z_0$ by Picard's Great Theorem.

\begin{proof}[Proof of Theorem \ref{thm:main}]
  We know from Theorem \ref{IntCrit} that $\mathcal{F}_C$ is quasi--normal of order at most $C$.
Let $(f_n)$ be a sequence in $\mathcal{F}_C$. Then the quasi--normality of
$\mathcal{F}_C$ implies that there exists a subsequence which we still denote
by $(f_n)$ and a finite set $S \subseteq D$ such that $(f_n)$ converges
locally uniformly on $D \setminus S$ to some limit function $f$ which is
meromorphic in $D \setminus S$ and such that $(f_n)$ fails to be normal in any
neighborhood of any point of $S$.

If $S=\emptyset$, then $f$ is meromorphic on $D$ and clearly also belongs to
$\mathcal{F}_C$, so alternative (1) in Theorem \ref{thm:main} holds.
We assume from now on that $S$ is not empty.
In order to prove that (2a) holds, let $p \in S$ and choose $\eps>0$
such that the closed disk $K$ of radius $\eps$ centered at $p$ does not contain any
other point of $S$. Then $(f_n)$ converges spherically uniformly on each annulus $\delta
\le |z-p| \le \eps$ to $f$, so
$$ \frac{1}{\pi}\int \limits_{K \setminus \{p\}} \big( f^{\sharp}(z) \big)^2 \, dxdy \le
C \, .$$
Therefore, $f$ has a meromorphic extension to $K$ by Lemma \ref{lem:heb}, and
hence to all of $D$. We denote this extension also by $f$. It follows at once that $f
\in \mathcal{F}_C$.

Since $(f_n)$ is not normal at $p$,  Marty's theorem \cite{Ma} yields  a
sequence $(z_n)$ in $D$ converging to $p$ such that $f_n^{\sharp}(z_n) \to +\infty$.

If each $f_n$ is locally univalent on $D$, then $f$ is either a constant
function or
also locally univalent in $D \setminus S$ by the locally uniform convergence of
$(f_n)$ to $f$ on $D
\setminus S$ and Hurwitz' theorem. If $f$ would be locally univalent on $D
\setminus S$, then
with $p\in S$ and $\eps>0$ as above, there would be a constant $c>0$ such that $f^{\sharp}(z)
\ge 2 c$ for all $|z-p|=\eps$. Hence $f_n^{\sharp}(z) \ge c$ for all
$|z-p|=\eps$ and all $n$ sufficiently large. Since each $f_n$ is locally
univalent on $D$, so $f_n^{\sharp}$ is never zero in D, the equation  $-\Delta
\log f^{\sharp}_n=4 \big( f_n^{\sharp} \big)^2$ shows that $\log f_n^{\sharp}$
is superharmonic on $|z-p| <\eps$ and thus $f^{\sharp}_n(z) \ge c$ for all 
 $|z-p| \le \eps$ by the minimum principle. But then
as we have already pointed out in Section \ref{NBLU},  the main result in \cite{GN} implies that $(f_n)$
is normal in $|z-p|<\eps$, contradicting $p \in S$. Hence, $f$ is in fact constant on $D$, and we have proved (2a).

We finally prove (2b). Since
$$ \left( \big( f^{\sharp}_n  \big)^2 \right) $$ is bounded in $L^1(D)$
we can apply the weak--$*$ compactness principle for bounded Borel measures
which yields a (unique) nonnegative bounded Borel measure $\mu$ on $D$ such that -- after taking
a subsequence -- we have
$$  \big( f^{\sharp}_n \big)^2 \to \mu$$
in the measure theoretic sense, that is,
\begin{equation} \label{eq:conv}
\int \limits_D \big( f^{\sharp}_n(z) \big)^2 \psi(z) \, dx dy \to \int
\limits_{D} \psi \, d\mu
\end{equation}
for any $\psi \in C_c(D)$. Here, $C_c(D)$ denotes the set of continuous
functions on $D$ with compact support. It follows that for any such $\psi$,
\begin{eqnarray*} \int \limits_D \psi \, d\mu &=& \sum \limits_{p \in S}
                                                  \left( \lim \limits_{n \to
  \infty} \int \limits_{K_{\eps}(p)}  \big( f^{\sharp}_n(z) \big)^2
\big(\psi(z)-\psi(p)\big) \, dxdy+\psi(p) \lim \limits_{n \to \infty} \int \limits_{K_{\eps}(p)}  \big(
  f^{\sharp}_n(z) \big)^2 \, dx dy \right)\\ & & + \lim \limits_{n \to \infty} \int \limits_{D \setminus
  \bigcup \limits_{p \in S} K_{\eps}(p)}  \big( f^{\sharp}_n(z) \big)^2 \, \psi(z)\, dx
                                                 dy 
\end{eqnarray*}
for any $\eps>0$ sufficiently small. Since $f_n \to f$ locally uniformy in $D \setminus
S$, $f_n \in \mathcal{F}_C$ and $\psi$ is continuous, this yields

$$ \int \limits_{D} \psi \, d\mu=\sum \limits_{p \in S} \mu(\{p\})  \psi(p)+\int
\limits_{D} \big(f^{\sharp}(z) \big)^2 \, \psi(z) \, dx dy \, .$$
Hence,
  $$ \mu=\sum \limits_{p \in S} \mu(\{p\}) \delta_{p}+\big( f^{\sharp}  \big)^2
  \,  $$
  in the sense of measures.
  Now, $\mu(\{ p\}) \ge \pi$ for each $p \in S$, since otherwise there would exist a
  function $\psi \in C_c(D)$ with  $0 \le \psi \le 1$ such that  $\psi \equiv 1$ on some disk
   centered at $p$  and
  $$ \int \limits_{D} \psi \, d\mu<\pi \, .$$
In view of (\ref{eq:conv}), this would imply that eventually
  $$ \int \limits_{K_{\eps}(p)} \big( f^{\sharp}_n(z) \big)^2 \, dxdy<\pi \, ,$$
  so $(f_n)$ would be normal on $K_{\eps}(p)$ by Theorem \ref{IntCrit},
  contradicting  $p \in S$. Hence $\alpha_p:=\mu(\{p\})/\pi \ge 1$.
\end{proof}

\section{Final Remark: Gromov's compactness theorem for pseudoholomorphic curves}

Theorem A and Theorem \ref{thm:main}  are   mere examples of a general
phenomenon. 
Thinking of
$$ \int \limits_{D} \big( f^{\sharp} (z) \big)^2 \, dx dy $$
as ``Energy'', the ``Mass Concentration Property'' (2b) tells us that a
discrete loss of energy when passing from $(f_n)$ to the limit $f$ has to be compensated
by  ``bubbles'' appearing as $n$ tends to infinity. This is a key idea in
Geometric Analysis, which is impressively demonstrated in  the proof of
Gromov's compactness theorem \cite{G} or  the work of Sacks and
Uhlenbeck \cite{SU} on minimal immersions of $2$-spheres.

\end{document}